\documentclass[reqno,11pt]{amsart}
\usepackage{amsmath,amsthm,amssymb,amsfonts,epsfig,color,graphicx,enumerate}
\usepackage{accents}
\usepackage{subfigure}

\DeclareMathOperator{\dive}{div}

\DeclareMathOperator{\per}{Per}
\DeclareMathOperator{\dist}{dist}

\def\ds{\displaystyle}
\def\eps{{\varepsilon}}
\def\N{\mathbb{N}}
\def\R{\mathbb{R}}
\def\O{\Omega}
\def\Ob{\overline{\Omega}}

\def\HH{\mathcal{H}}

\def\LL{\mathcal{L}}

\newcommand{\be}{\begin{equation}}
\newcommand{\ee}{\end{equation}}
\newcommand{\bib}[4]{\bibitem{#1}{\sc#2: }{\it#3. }{#4.}}

\newcommand{\s}{\sigma}
\newcommand{\sint}{\s_{{\rm{int}}}}
\newcommand{\sext}{\s_{{\rm{ext}}}}
\newcommand{\uint}{u_{{\rm{int}}}}
\newcommand{\uext}{u_{{\rm{ext}}}}

\numberwithin{equation}{section}
\theoremstyle{plain}
\newtheorem{teo}{Theorem}[section]

\newtheorem{prop}[teo]{Proposition}

\theoremstyle{remark}
\newtheorem{oss}[teo]{Remark}
\newtheorem{exam}[teo]{Example}

\newenvironment{ack}{{\bf Acknowledgements. }}

\title[Optimal regions for congested transport]{Optimal regions for congested transport}

\author{Giuseppe Buttazzo}
\address{Dipartimento di Matematica, Universit\`a di Pisa, Largo B. Pontecorvo 5, 56126 Pisa, ITALY}
\email{buttazzo@dm.unipi.it}

\author{Guillaume Carlier}
\address{CEREMADE UMR CNRS 7534, Universit\'e de Paris Dauphine, Pl.~de Lattre de Tassigny, 75775 Paris Cedex 16, FRANCE}
\email{carlier@ceremade.dauphine.fr}

\author{Serena Guarino Lo Bianco}
\address{Dipartimento di Matematica, Universit\`a di Pisa, Largo B. Pontecorvo 5, 56127 Pisa, ITALY}
\email{sguarino@mail.dm.unipi.it}

\begin{document}
\maketitle

\begin{abstract}
We consider a given region $\O$ where the traffic flows according to two regimes: in a region $C$ we have a low congestion, where in the remaining part $\O\setminus C$ the congestion is higher. The two congestion functions $H_1$ and $H_2$ are given, but the region $C$ has to be determined in an optimal way in order to minimize the total transportation cost. Various penalization terms on $C$ are considered and some numerical computations are shown.
\end{abstract}



\section{Introduction}\label{sintro}

As everybody has experienced while traveling in urban traffic, the planning of an efficient network of roads is an extremely complex problem, involving many parameters such as the distribution of residences and working places, the period of the day one travels, the attitude of drivers, \dots. The congestion effects are also to be taken into account, since they are responsible for the formation of traffic jams and and have a social cost in terms of waste of time.

In the present paper we consider a very simplified model in which the densities of residents and of working places are known, represented by two probability measures $f^+$ and $f^-$. The congestion effects in mass transportation theory has been deeply studied in the literature; we refer for instance to \cite{cajisa, bracar} and references therein. Denoting by $f$ the difference $f=f^+-f^-$ and by $\s$ the traffic flux, the model, in the stationary regime, reduces to a minimization problem of the form
\be\label{congest}
\min\left\{\int_\O H(\s)\,dx\ :\ -\dive \s=f\hbox{ in }\O,\ \s\cdot n=0\hbox{ on }\partial\O\right\}.
\ee
Here $\O$ is the urban region under consideration, a bounded Lipschitz domain of $\R^d$, the boundary conditions at $\partial\O$ are usually taken imposing zero normal flux $\s\cdot n=0$, and $H:\R^d\to[0,+\infty]$ is the {\it congestion function}, a convex nonnegative function with $\lim_{|s|\to+\infty}H(s)=+\infty$. The first order PDE
$$-\dive \s=f\hbox{ in }\O,\qquad \s\cdot n=0\hbox{ on }\partial\O$$
has to be intended in the weak sense
$$\langle \s,\nabla\phi\rangle=\langle f,\phi\rangle\qquad\hbox{for every } \phi\in C^\infty(\Ob)$$
and it captures the equilibrium between the traffic flux $\s$ and the difference between supply and demand $f$.

In the case $H(s)=|s|$ no congestion effect occurs, and the transport problem reduces to the Monge's transport, where mass particles travel along geodesics (segments in the Euclidean case). As it is well known, in the Monge's case the integral cost above is finite for every choice of the probabilities $f^+$ and $f^-$. On the contrary, when $H$ is superlinear, that is
$$\lim_{|s|\to+\infty}\frac{H(s)}{|s|}=+\infty,$$
congestion effects may occur and the mass particles trajectories follow more complicated paths. In this case the integral cost can be $+\infty$ if the source and target measures $f^+$ and $f^-$ are singular. For instance, if the congestion function $H$ has a quadratical growth, in order to have a finite cost it is necessary that the signed measure $f=f^+-f^-$ be in the dual Sobolev space $H^{-1}$; thus, if $d>1$ and the measures $f^+$ or $f^-$ contain some Dirac mass, the minimization problem \eqref{congest} is meaningless. In other words, superlinear congestion costs prevent too high concentrations.

In the present paper, we aim to address the efficient design of low-congestion regions; more precisely, two congestion functions $H_1$ and $H_2$ are given, with $H_1\le H_2$, and the goal is to find an optimal region $C\subset\O$ where the less congested traffic may travel. Since reducing the congestion in a region $C$ is costly (because of roads improvement, traffic devices, \dots), a term $m(C)$ will be added, to describe the cost of improving the region $C$, then penalizing too large low-congestion regions. On the region $\O\setminus C$ we then have the normally congested traffic governed by the function $H_2$, while on the low-congestion region $C$ the traffic is governed by the function $H_1$. Throughout the paper, we will assume that $H_1$ and $H_2$ are two continuous convex functions such that $0\le H_1\le H_2$ and 
$$\lim_{|s|\to+\infty}\frac{H_i(s)}{|s|}=+\infty, \; i=1,2.$$

The mathematical formulation of the problem then is as follows, for every region $C$ we may consider the cost function
\be\label{costf}
F(C)=\min\left\{\int_{\O\setminus C}H_2(\s)\,dx+\int_C H_1(\s)\,dx\ :\ -\dive \s=f\hbox{ in }\O,\ \s\cdot n=0\hbox{ on }\partial\O\right\},
\ee
so that the optimal design of the low-congestion region amounts to the minimization problem
\be\label{minpb}
\min\big\{F(C)+m(C)\ :\ C\subset\O\big\}.
\ee

Several cases will be studied in the sequel, according to the various constraints on the low-congestion region $C$ and the corresponding penalization/cost $m(C)$. The simplest case is when $C$ is a $d$-dimensional subdomain of $\O$ and the penalization $m(C)$ involves the perimeter of $C$: in this situation an optimal region $C$ is shown to exist and a necessary optimal condition is established.

When $m(C)$ is simply proportional to the Lebesgue measure of $C$ (that we will denote by $\LL^d(C)$ or by $|C|$), on the contrary an optimal domain $C$ may fail to exist and a relaxed formulation of the problem has to be considered; in this case the optimal choice for the planner is to have a low-congestion area $C_0$, a normally congested area $C_1$, together with an area $\O\setminus(C_0\cup C_1)$ with intermediate congestion (that is a mixing of the two congestion functions occurs). For this case, we also give some numerical simulations in dimension two.

Another class of problems arises when the admissible sets $C$ are {\it networks}, that is closed connected one-dimensional sets. In this case the penalization cost $m(C)$ is proportional to the total lenght of $C$ (the 1-dimensional Hausdorff measure $\HH^1(C)$). In this case the precise definition of the cost function $F(C)$ in \eqref{costf} has to be reformulated more carefully in terms of measures (see Section \ref{snet}). This one-dimensional problem has been extensively studied in the extremal case where $H_1=0$ and $H_2(s)=|s|$ (see for instance \cite{bms09, bos02, bpss09, bust03, bust04}) providing an interesting geometric problem called {\it average distance problem}; an extended review on it can be found in \cite{lem12}. We also point out that a similar problem arises in some models for the mechanics of damage, see for instance \cite{bcgs}.

\section{Perimeter constraint on the low-congestion region}\label{sperim}

In this section we consider the minimum problem \eqref{minpb}, where the cost $F(C)$ is given by \eqref{costf} and $m(C)=k\per(C)$, being $k>0$ and $\per(C)$ the perimeter of the set $C$ in the sense of De Giorgi (see for instance \cite{afp00}). Thanks to the coercivity properties of the perimeter with respect to the $L^1$ convergence of the characteristic functions (that we still call $L^1$ convergence of sets), we have the following existence result.

\begin{teo}\label{exper}
Assume that the cost $F(C)$ is finite for at least a subset $C$ of $\Ob$ with finite perimeter and that $m(C)=k\per(C)$ with $k>0$. Then there exists at least an optimal set $C_{opt}$ for problem \eqref{minpb}.
\end{teo}

\begin{proof}
Let $(C_n)_{n\in\N}$ be a minimizing sequence for the optimization problem \eqref{minpb}; then sequence $\per(C_n)$ is bounded. Thanks to the compactness of the embedding of ${\rm{BV}}$ into $L^1$, we may extract a (not relabeled) subsequence converging in $L^1$ to a subset $C$ of $\O$. We claim that this set $C$ is an optimal set for the problem \eqref{minpb}. Indeed, for the properties of the perimeter we have
$$\per(C)\le\liminf_n\per(C_n).$$
Moreover, if we denote by $\s_n$ an optimal (or asymptotically optimal) function for
$$F(C_n)=\int_{\O\setminus C_n}H_2(\s_n)\,dx+\int_{C_n}H_1(\s_n)\,dx$$
with
$$-\dive\s_n=f\hbox{ in }\O,\qquad\s_n\cdot n=0\hbox{ on }\partial\O,$$
by the superlinearity assumption on the congestion functions $H_1$ and $H_2$, and by the De La Vall\'ee Poussin compactness theorem, we have that $(\s_n)_{n\in\N}$ is compact for the weak $L^1$ convergence and so we may assume that $\s_n$ weakly converges in $L^1(\O)$ to a suitable function $\s$. This function $\s$ still verifies the condition
$$-\dive\s=f\hbox{ in }\O,\qquad\s\cdot n=0\hbox{ on }\partial\O.$$
Thanks to the convexity of $H_1$ and $H_2$ and the strong-weak lower semicontinuity theorem for integral functionals (see for instance \cite{bu89}), we have
$$F(C)\le\liminf_n F(C_n).$$
Therefore the set $C$ is optimal and the proof is concluded.
\end{proof}

Our aim now is to establish optimality conditions not only on an optimal flow $\s$ but also on the corresponding optimal low-congestion regions $C$. Optimality conditions for $\s$ can be directly derived from the duality formula:
\[\begin{split}
F(C)&=\inf_{\s\in\Gamma_f}\int_{C}H_1(\s)\,dx+\int_{\O\setminus C}H_2(\s)\,dx\\
&=-\inf_{u}\Big\{\int_C H_1^*(\nabla u)\,dx+\int_{\O\setminus C}H_2^*(\nabla u)\,dx-\int_\O uf\,dx\Big\},
\end{split}\]
from which one easily infers that 
\[\s=\begin{cases}
\sint\mbox{ in }C,\\
\sext\mbox{ in }\O\setminus C
\end{cases}\]
where 
\[\sint=\nabla H_1^*(\nabla\uint)\mbox{ in }C,\qquad\sext=\nabla H_2^*(\nabla\uext),\mbox{ in }\O\setminus C\]
the minimizer $u$ in the dual is then given by:
\[u=\begin{cases}
\uint\mbox{ in }C,\\
\uext\mbox{ in }\O\setminus C
\end{cases}.\]
We have used the notations $\sint$, $\sext$, $\uint$ and $\uext$ to emphasize the fact that $\s$ and $\nabla u$ may have a discontinuity when crossing $\partial C$. It is reasonable (by elliptic regularity and assuming smoothness of $C$) to assume that $\s$ and $\nabla u$ are Sobolev on $C$ and $\O\setminus C$ separately but they are a priori no better than $BV$ on the whole of $\O$. The functions $\uint$ and $\uext$ are then at least formally characterized by the Euler-Lagrange equations
\[-\dive\Big(\nabla H_1^*(\nabla\uint)\Big)=f\mbox{ in }C,\qquad
-\dive\Big(\nabla H_2^*(\nabla\uext)\Big)=f,\mbox{ in }\O\setminus C\]
together with
\[\nabla H_1^*(\nabla\uint)\cdot n=0,\mbox{ on }\partial\O\cap C,\qquad\nabla H_2^*(\nabla\uext)\cdot n=0,\mbox{ on }\partial\O\cap\Ob\setminus C,\]
and (assuming that $f$ does not give mass to $\partial C$) the continuity of the normal component of $\sigma$ across $\partial C$:
\[\Big(\nabla H_1^*(\nabla\uint)-\nabla H_2^*(\nabla\uext)\Big)\cdot n_C=0,\mbox{ on }\partial C\cap\O,\]
where $n_C$ denotes the exterior unit vector to $C$.

Now, we wish to give an extra optimality condition on $C$ itself assuming that is smooth. To do so, we take a smooth vector field $V$ such that $V\cdot n=0$ on $\partial\O$, and we set $C_t=\varphi_t(C)$, where $\varphi_t$ denotes the flow of $V$ (i.e. $\varphi_0={\rm{id}}$, $\partial_t\varphi_t(x)=V(\varphi_t(x))$). For $t>0$, we then have
\be\label{ineg0}
0\le\frac{1}{t}[F(C_t)-F(C)+k\per(C_t)-k\per(C)].
\ee
As for the perimeter term, it is well-known (see for instance \cite{hepi05}) that the first-variation of the perimeter involves the mean curvature $\HH$ of $\partial C$, more precisely, we have:
\be\label{per}
\frac{d}{dt}\per(C_t)\big|_{t=0}=\int_{\partial C}\HH\,V \cdot n_C\,d\HH^{d-1}.
\ee
For the term involving $H$, we observe that
\[F(C_t)-F(C)\le\int_{C_t}H_1(\s)\,dx-\int_C H_1(\s)\,dx+\int_{\O\setminus C_t }H_2(\s)\,dx-\int_{\O\setminus C}H_2(\s)\,dx.\] 
At this point, we have to be a little bit careful because of the discontinuity of $\s$ at $\partial C$, but distinguishing the part of $\partial C$ on which $V\cdot n_C>0$ that is moved outside $C$ by the flow, and that on which $V\cdot n_C<0$ that is moved inside $C$ by the flow, and arguing as in Theorem 5.2.2 of \cite{hepi05}, we arrive at:
\be\label{inegF}
\begin{split}
\limsup_{t\to0}\frac{F(C_t)-F(C)}{t}\le&\int_{\partial C}\Big(\big(H_1(\sext)-H_2(\sext)\big)(V\cdot n_C)_+\\
&+\big(H_2(\sint)-H_1(\sint)\big)(V\cdot n_C)_-\Big)\,d\HH^{d-1}.
\end{split}\ee
Combining \eqref{ineg0}, \eqref{per} and \eqref{inegF}, we obtain
$$0\le\int_{\partial C}\Big(\big(H_1(\sext)-H_2(\sext)+k\HH\big)(V\cdot n_C)_+
+\big(H_2(\sint)-H_1(\sint)-k\HH\big)(V\cdot n_C)_-\Big)\,d\HH^{d-1}.$$
But since $V$ is arbitrary, we obtain the extra optimality conditions:
$$H_2(\sint)-H_1(\sint)\ge k\HH\ge H_2(\sext)-H_1(\sext)
\qquad\mbox{ on }\partial C\cap\O$$
which, since $H_2\ge H_1$, in particular implies that $\partial C$ has nonnegative mean curvature. 

The regularity of $\partial C$ is an interesting open question. Note that when $d=2$ and $\O$ is convex, replacing $C$ by its convex hull diminishes the perimeter and also the congestion cost, so that optimal regions $C$ are convex, this is a first step towards regularity, note also that convexity of optimal regions is consistent with  the curvature inequality above. 

Let us illustrate the previous conditions on the simple quadratic case where $H_1(\s)=\frac{a}{2}|\s|^2$, $H_2(\s)=\frac{b}{2}|\s|^2$ with $0<a<b$. The optimality conditions for the pair $u$, $\sigma$ then read as
\[\begin{cases}
-a\Delta\uint=f&\mbox{ in } C\\
-b\Delta\uext=f&\mbox{ in }\O\setminus C,
\end{cases}
\qquad\frac{\partial u}{\partial n}=0\mbox{ on $\partial\O$},
\qquad\begin{cases}
\sint=\frac{\nabla\uint}{a}\\
\sext=\frac{\nabla\uext}{b},
\end{cases}\]
together with
\[\Big(\frac{\nabla\uint}{a}-\frac{\nabla\uext}{b}\Big)\cdot n_C=0
\qquad\mbox{ on }\partial C\cap\O\]
and
\[\frac{b-a}{2}|\sint|^2=\frac{b-a}{2a^2}|\nabla\uint|^2\ge k\HH
\ge\frac{b-a}{2}|\sext|^2=\frac{b-a}{2b^2}|\nabla\uext|^2
\qquad\mbox{ on }\partial C\cap\O\]
where $\HH$ again denotes the mean curvature of $\partial C$.

\section{Relaxed formulation for the measure penalization}\label{srelax}

In this section we consider the case when the penalization on the low-congestion region is proportional to the Lebesgue measure, that is $m(C)=k|C|$ with $k>0$. The minimization problem we are dealing with then becomes
\be\label{measureconstraint}
\min_{\s,C}\left\{\int_C H_1(\s)\,dx+\int_{\O\setminus C}H_2(\s)\,dx+k|C|\ :\ \s\in\Gamma_f\right\}\ee
where
$$\Gamma_f=\big\{\s\in L^1(\O;\R^d)\ :\ \dive \s=f\hbox{ in }\O,\ \s\cdot n=0\hbox{ on }\partial\O\big\}.$$
Passing from sets $C$ to density functions $\theta$ with $0\le\theta(x)\le1$ we obtain the relaxed formulation of \eqref{measureconstraint}
\be\label{relax}
\min_{\s,\theta}\left\{\int_\O\theta H_1(\s)\,dx+\int_\O(1-\theta)H_2(\s)\,dx+k\int_\O\theta\,dx\ :\ \s\in\Gamma_f\right\}.
\ee
Writing the quantity to be minimized as
$$\int_\O H_2(\s)+\theta\big(H_1(\s)+k-H_2(\s)\big)\,dx$$
the minimization with respect to $\theta$ is straightforward; in fact, if $H_1(\s)+k>H_2(\s)$ we take $\theta=0$, while if $H_1(\s)+k<H_2(\s)$ we take $\theta=1$. In the region where $H_1(\s)+k=H_2(\s)$ the choice of $\theta$ is irrelevant. In other words, we have taken
$$\theta=1_{\{H_1(\s)+k<H_2(\s)\}},$$
which gives
$$H_2+\theta\big(H_1+k-H_2\big)=H_2-\big(H_1+k-H_2\big)^-=H_2\wedge\big(H_1+k\big).$$
Therefore, in the relaxed problem \eqref{relax} the variable $\theta$ can be eliminated and the problem reduces to
\be\label{relaxinterm}
\min\left\{\int_\O H_2(\s)\wedge\big(H_1(\s)+k\big)\,dx\ :\ \s\in\Gamma_f\right\}.\ee
Clearly the infimum in \eqref{relaxinterm} coincides with that of \eqref{measureconstraint} but since the new integrand $H_2\wedge\big(H_1+k\big)$ is not convex, a further relaxation with respect to $\s$ is necessary. This relaxation issue with a divergence constraint has been studied in \cite{bra87}, where it is shown that the relaxation procedure amounts to convexify the integrand. We then end up with the minimum problem
\be\label{relaxpb}
\min\left\{\int_\O\Big(H_2(\s)\wedge\big(H_1(\s)+k\big)\Big)^{**}\,dx\ :\ \s\in\Gamma_f\right\}
\ee
where $**$ indicates the convexification operation. Recalling that $H_1$ and $H_2$ are superlinear, we have that:
\begin{itemize}
\item[-]in the region where
$$\Big(H_2\wedge\big(H_1+k\big)\Big)^{**}(\s)=H_2(\s)$$
we take $\theta=0$. In other words, in this region, it is better not to spend resources for improving the traffic congestion;

\item[-]in the region where
$$\Big(H_2\wedge\big(H_1+k\big)\Big)^{**}(\s)=H_1(\s)+k$$
we take $\theta=1$. In other words, in this region, it is necessary to spend a lot of resources for improving the traffic congestion;

\item[-]in the region where
$$\Big(H_2\wedge\big(H_1+k\big)\Big)^{**}(\s)<\Big(H_2\wedge\big(H_1+k\big)\Big)(\s)$$
we have $0<\theta(x)<1$ so that there is some mixing between the low and the high congestion functions. In other words, in this region the resources that are spent for improving the traffic congestion are proportional to $\theta$.
\end{itemize}

The previous situation is better illustrated in the case where both functions $H_1$ and $H_2$ depend on $|\s|$ and $H_2-H_1$ increases with $|\s|$. In this case, we denote by $r_1$ the maximum number such that
$$\Big(H_2\wedge\big(H_1+k\big)\Big)^{**}(r)=H_2(r)$$
and by $r_2$ the minimum number such that
$$\Big(H_2\wedge\big(H_1+k\big)\Big)^{**}(r)=H_1(r)+k,$$
then we have
$$\theta(x)=\frac{|\s|-r_1}{r_2-r_1}\qquad\hbox{whenever }r_1<|\s|<r_2.$$
In this case, for small values of the traffic flow ($|\s|\le r_1$), it is optimal not to spend any resource to diminish congestion, on the contrary when traffic becomes large ($|\s|\ge r_2$), it becomes optimal to reduce the congestion to $H_1$. Finally, for intermediate values of the traffic, mixing occurs with the coefficient $\theta$ above as a result of the relaxation procedure. 

Also, problem \eqref{relaxpb} is of type \eqref{congest} and it is well-known, by convex analysis, that we have the dual formulation
\begin{align}\label{dual}
\min\Big\{\int_\O H(\s)\,dx\ :\ \s\in\Gamma_f\Big\}&=\sup\Big\{\int_\O u\,df-\int_\O H^*(\nabla u)\,dx\Big\}\nonumber\\
&=-\inf\Big\{\int_\O H^*(\nabla u)\,dx-\int_\O u\,df\Big\},
\end{align}
where $H(\s)=(H_2(\s)\wedge(H_1(\s)+k))^{**}$.
Notice that the Euler-Lagrange equation of problem \eqref{dual} is formally written as
\be\label{lap}
\begin{cases}
-\dive\nabla H^*(\nabla u)=f&\mbox{in }\O\\
\nabla H^*(\nabla u)\cdot\nu=0&\mbox{on }\partial\O.
\end{cases}
\ee
Moreover, the link between the flux $\s$ and the dual variable $u$ is
$$\s=\nabla H^*(\nabla u).$$
In our case, the Fenchel tranform is easy computed and we have:
$$H^*(\xi)=H_2^*(\xi)\vee(H_1^*(\xi)-k).$$

As a conclusion of this paragraph, we observe that the treatment above is similar to the analysis of two-phase optimization problems. This consists in finding an optimal design for a domain that is occupied by two constituent media with constant conductivities $\alpha$ and $\beta$ with $0<\alpha<\beta<+\infty$, under an objective function and a state equation that have a form similar to \eqref{dual} and \eqref{lap}. We refer to \cite{bubu05} (and references therein) for a general presentation of shape optimization problems and to \cite{allaire} for a complete analysis of two-phase optimization problems together with numerical methods to treat them.

\section{low-congestion transportation networks}\label{snet}

In this section, our main unknown is a one-dimensional subset $\Sigma$ of $\O$; we consider a fixed number $r>0$ and the low-congestion regions of the form
$$C_{\Sigma,r}=\big\{x\in\Ob\ :\ \dist(x,\Sigma)\le r\big\}
=\Sigma^r\cap\Ob,\mbox{ where }\Sigma^r:=\Sigma+B_r(0).$$
and $\Sigma$ is required to be a closed subset of $\Ob$ such that $\HH^1(\Sigma)<+\infty$. The penalization term $m(C_{\Sigma,r})$ is taken proportional to the Lebesgue measure of $C_{\Sigma,r}$, so that our optimization problem becomes
\be\label{pbsigmar}
\min_{\s,\Sigma}\left\{\int_{C_{\Sigma,r}}H_1(\s)\,dx+\int_{\O\setminus C_{\Sigma,r}}H_2(\s)\,dx+k|C_{\Sigma,r}|\ :\ \s\in\Gamma_f\right\}
\ee
with $k>0$. A key point in the existence proof below consists in remarking that the perimeter of an $r$-enlarged set $\Sigma^r$ can be controlled by its measure (see Appendix \ref{Aa}). It also worth remarking that $\Sigma^r$ has the \emph{uniform interior ball of radius $r$ property}; for every $x\in\Sigma^r$ there exists $y\in \R^d$ such that $|x-y|\le r$ and $B_r(y)\subset\Sigma^r$. Clearly, $r$-enlarged sets have the uniform interior ball of radius $r$ property and sets with this property are $r$-enlarged sets (i.e. can be written as the sum of a closed set and $B_r(0)$), we refer to \cite{acm} for more on sets with the uniform interior ball property, and in particular estimates on their perimeter (which we could have used instead of the more elementary Lemma in Appendix \ref{Aa}).

\begin{prop}\label{exsigmar}
Ler $r>0$ be fixed, $d=2$ and assume that $F(C_{\Sigma,r})<+\infty$ for some closed one-dimensional subset $\Sigma$ of $\Ob$. Then the optimization problem \eqref{pbsigmar} admits a solution.
\end{prop}

\begin{proof}
The sets $C_{\Sigma,r}$ satisfy the inequality (see for instance Appendix \ref{Aa})
$$\per(C_{\Sigma,r})\le\frac{K}{r}|C_{\Sigma,r}|$$
for a suitable constant $K$ depending only on the dimension $d$. Therefore, for a minimizing sequence $(\Sigma_n)_{n\in\N}$, the sets $C_n:=C_{\Sigma_n,r}=\Sigma_n^r \cap\Ob$ are compact in the strong $L^1$ convergence, we can thus extract a (not relabeled) subsequence such that $C_n$ converges strongly in $L^1$ (and a.e.) to some $C$. One can then repeat the proof of Theorem \ref{exper}, to obtain
\[F(C)+k|C|\le\inf\eqref{pbsigmar}.\]
It only remains to show that $C$ can be obtained as $C=C_{\Sigma,r}$ (up to a negligible set) for some closed subset of $\Ob$, $\Sigma$ such that $\HH^1(\Sigma)<+\infty$. Up to an extraction, one can assume that $\Sigma_n^r$ converges for the Hausdorff distance to some compact set $E$ (which also satisfies the uniform interior ball property of radius $r$). Let us first check that $C=E\cap\Ob$ (up to a negligible set), the inclusion $C\subset E\cap\Ob$ is standard (see for instance \cite{hepi05}). To prove the converse inclusion, it is enough to show that $|C|=|E\cap\Ob|$ i.e. $|C_n|\to|E\cap\Ob|$ as $n\to\infty$. For this, we observe that
\[\big||C_n|-|E\cap\Ob|\big|\le|\Sigma_n^r\setminus E|+|E\setminus\Sigma_n^r|\]
The convergence of $|\Sigma_n^r\setminus E|$ to $0$ easily follows from the Hausdorff convergence of $\Sigma_n^r$ to $E$ and the fact that $E$ is closed (see \cite{hepi05} for details). As for the convergence of $|E\setminus\Sigma_n^r|$ to $0$, we proceed as follows: let $\eps>0$ and $n$ be large enough so that $E\subset \Sigma_n^r+B_\eps(0)=\Sigma_n^{r+\eps}$, using the coarea formula, as in Appendix \ref{Aa}, and the fact that the sets $\Sigma_n^{r+\eps}$ have bounded perimeter (see again Appendix \ref{Aa} or \cite{acm}), we get for $n$ large enough
$$|E\setminus\Sigma_n^r|\le|\Sigma_n^r+B_\eps(0)\setminus\Sigma_n^r|\le C\eps.$$
We thus have proved that $C=E\cap\Ob$ (up to a negligible set). Let us finally denote by $d$ the distance to $\R^2\setminus E$ and set
$$\Sigma:=\bigcup_{l=1}^L d^{-1}(\{lr\})$$
where $L$ is the integer part of $r^{-1}\max d$. It is then not difficult to check that $\HH^1(\Sigma)<+\infty$ and $\Sigma^r=E$ because $E$ satisfies the uniform interior ball property of radius $r$ so that $C=C_{\Sigma,r}$, which ends the proof.
\end{proof}

\begin{oss}
We have used the assumption that $d=2$ only in the last step that is to prove that $C=C_{\Sigma,r}$ for some \emph{one-dimensional} $\Sigma$. In higher dimensions, the same proof works if one requires $\HH^{d-1}(\Sigma)<+\infty$ (however we believe the result remains true for one-dimensional sets in any dimension). 
\end{oss}

\begin{oss}
If the admissible sets $\Sigma$ are supposed connected (in this case we call them {\it networks}), or with an a priori bounded number of connected components, then the penalization term $|C_{\Sigma,r}|$ can be replaced by the one-dimensional Hausdorff measure $\HH^1(\Sigma)$. In fact, for such sets we have
$$|C_{\Sigma,r}|\le M\big(1+\HH^1(\Sigma)\big)$$
where the constant $M$ depends on the dimension $d$, on $r$, and on the number of connected components of $\Sigma$. Therefore the argument of Proposition \ref{exsigmar} applies, providing the existence of an optimal solution.
\end{oss}

We deal now with the case when the low-congestion region is a one-dimensional set $\Sigma$. We assume $\Sigma$ connected (or with an a priori bounded number of connected components) and we take $m(\Sigma)$ proportional to the one-dimensional Hausdorff measure $\HH^1(\Sigma)$. The integral on the low-congestion region has to be modified accordingly and we have to consider the problem formally written as
\be\label{pbsigma0}
\min_{\s,\Sigma}\left\{\int_\Sigma H_1(\s)\,d\HH^1+\int_\O H_2(\s)\,dx+k\HH^1(\Sigma)\ :\ \s\in\Gamma_f\right\}
\ee
with $k>0$. Notice that, in view of the superlinearity assumption on the congestion functions $H_1$ and $H_2$, the admissible fluxes $u$ have to be assumed absolutely continuous measures with respect to $\LL^d\lfloor\O+\HH^1\lfloor\Sigma$. Subsequently, the integral terms in the cost expression have to be intended as:
$$\int_\Sigma H_1\Big(\frac{d\s}{d\HH^1}\Big)\,d\HH^1+\int_\O H_2\Big(\frac{d\s}{d\LL^d}\Big)\,dx.$$
By an abuse of notation, when no confusion may arise, we continue to write the terms above as $\int_\Sigma H_1(\s)\,d\HH^1+\int_\O H_2(\s)\,dx$.

In general, the optimization problem \eqref{pbsigma0} does not admit a solution $\Sigma_{opt}$, because the limits of minimizing sequences $\Sigma_n$ may develop multiplicities, providing as an optimum a relaxed solution made by a one-dimensional set $\Sigma_{opt}$ and function $a\in L^1(\Sigma_{opt})$ with $a(x)\ge1$. The relaxed version of problem \eqref{pbsigma0}, taking into account these multiplicities, becomes
\be\label{relpbsigma}
\min_{\s,\Sigma,a}\left\{\int_\Sigma H_1(\s/a)a\,d\HH^1+\int_\O H_2(\s)\,dx+k\int_\Sigma a\,d\HH^1\ :\ \s\in\Gamma_f\right\}.
\ee
The optimization with respect to $a$ is easy: consider for simplicity the case
$$H_1(\s)=\alpha|\s|^p\qquad\hbox{with }\alpha>0,\ p>1;$$
then we have
$$\min_{a\ge1}\left(ka+\alpha\frac{|\s|^p}{a^{p-1}}\right)=H(\s)=
\begin{cases}
\alpha|\s|^p+k&\ds\hbox{if }|\s|^p\le\frac{k}{\alpha(p-1)}\\
\ds|\s|\alpha^{1/p}p\Big(\frac{k}{p-1}\Big)^{1-1/p}&\ds\hbox{if }|\s|^p\ge\frac{k}{\alpha(p-1)}.
\end{cases}$$
Therefore the relaxed problem \eqref{relpbsigma} can be rewritten as
$$\min_{\s,\Sigma}\left\{\int_\Sigma H(\s)\,d\HH^1+\int_\O H_2(\s)\,dx\ :\ \s\in\Gamma_f\right\}$$
and the multiplicity density $a(x)$ on $\Sigma$ (that can be interpreted as the width of the road $\Sigma$ at the point $x$) is given by
$$a(x)=1\vee|\s(x)|\Big(\frac{\alpha(p-1)}{k}\Big)^{1/p}.$$

\section{Numerical simulations}

Here we wish to give a numerical example which clarifies and confirms what we expected from the analysis done in Section \ref{srelax}. In our examples, we mainly focus on the problem in the form \eqref{dual}:
$$\min\Big\{\int_\O H^*(\nabla u)\,dx-\int_\O fu\,dx\Big\}.$$

The numerical simulation is based on a very simple situation that however seems quite reasonable. The two congestion function considered are both quadratic but with a different coefficient, say $H_1(\s)=a|\s|^2$ and $H_2(\s)=b|\s|^2$ with $a<b$. Then, in this case, the function $H^*$ involved in \eqref{dual} is easy to compute:
$$H^*(\xi)=\Big(\frac{\xi^2}{4b}\Big)\vee\Big(\frac{\xi^2}{4a}-k\Big)$$

Before we start illustrating the numerical result, it is useful to do some considerations that justify the choice of some parameters in the following. The dual variable $u$ has to be thought as a price system for a company handling the transport in a congested situation. An optimizer $u$ then gives the price system which maximizes the profit of the company. When you take into account a congested transport between sources (here called $f^+$ and $f^-$), the total mass plays an important role: as observed in \cite{bracar}, in the case of a small mass, hence of a large Lagrange multiplier $k$, the congestion effects are negligible, so one can expect in this case a distribution of the low-congestion region around the source distribution. On the contrary, for a large mass, hence for a small Lagrange multiplier $k$, we may expect a distribution of the low-congestion region also between the sources $f^+$ and $f^-$.

In the following examples, we will consider as sources $f^+$ and $f^-$ two Gaussian distributions with variance $\lambda$, centered at two points $x_0$ and $x_1$
$$f^+(x)=\frac{1}{\sqrt{2\pi\lambda}}e^{-|x-x_0|^2/(2\lambda)},\qquad f^-(x)=\frac{1}{\sqrt{2\pi\lambda}}e^{-|x-x_1|^2/(2\lambda)}.$$
Of course, a large value of $\lambda$ means less concentration (and, on  the contrary, a small $\lambda$ captures more concentration); analogously, a large value of the penalization parameter $k$ corresponds to a small quantity of available resources. Ending this consideration on parameters involved, we note that the traffic congestion parameters $a,b$ and the ``construction cost'' parameter $k$ are linked: we will change value of $k$ according to a suitable choice of ratio $\frac{a}{b}$, for fixed $\lambda$. Now, concerning the choice of the coefficients $a,b$ we take $a=1$ and $b=4$, which means that the velocity in the low-congestion region is, at equal traffic density, four times the one in the region with normal congestion.

Using the equivalent dual formulation \eqref{dual} of problem \eqref{relax}, we find numerically the solution $u$, hence the flux $\s$ and the optimal density $\theta$.

Now, using the dual formulation of the problem, we find numerically the solution $u$ of \eqref{dual} and we obtain the flux $\s$ as explained in Section \ref{srelax}. The numerical procedure  to find $u$ uses a \textit{Quasi-Newton} method that updates an approximation of  the Hessian matrix at each iteration (see \cite{bfgs} and reference therein). First we generate a finite element space with respect to a square grid. Then we implement the BFGS method, using a routine included in the packages of software {\tt FreeFem3D} (available at \textit{http://www.freefem.org/ff3d}) that has the follow structure:
$$\hbox{BFGS(J,dJ,u,eps=1.e-6,nbiter=20) }$$

The routine above means: \textit{find the optimal ``u'' for the functional J}. The necessary parameters are the functional $J$, the gradient $dJ$ and the $u$ variable. Other parameters are optional and with clear means.

\begin{exam}
The common setting of the simulation is a transportation domain $\O=[0,1]^2$ with a $30\times30$ grid; we consider as initial and final distribution of resources two Gaussian approximations (with common variance $\lambda$) of Dirac delta function $f^-$ and $f^+$ respectively centered at $x_0=(0.3,0.3)$ and $x_1=(0.7,0.7)$. In the examples below we take different values of the parameters $k$ and $\lambda$ according to the considerations above, to show how the optimal distributions of the low-congestion regions may vary. Using the same notation as in Section \ref{srelax}, there are black and white region (respectively $\theta=1$ and $\theta=0$), passing through grey levels for the intermediate congestion.

In Figure \ref{fig1} we take the variance parameter $\lambda=0.02$, which provides the initial and final mass distributions not too concentrated, as depicted in Figure \ref{fig1} (a). In Figure \ref{fig1} (b) and (c) we take the penalization parameter $k=0.06$ and $k=0.4$ respectively; we see that in these cases, due to the low concentration of the initial and final mass distributions, the optimal density $\theta$ is higher in the region between $x_0$ and $x_1$.

In Figure \ref{fig2} we take the variance parameter $\lambda=0.001$, which provides the initial and final mass distributions rather concentrated, as depicted in Figure \ref{fig2} (a). In Figure \ref{fig2} (b) and (c) we take the penalization parameter $k=0.01$ and $k=0.05$ respectively; we see that in these cases, due to the high concentration of the initial and final mass distributions, the optimal density $\theta$ is higher in the region around $x_0$ and $x_1$. 

\begin{figure}[h!]
\centering%
\subfigure[$\lambda=0.02$]%
{\includegraphics[scale=0.33]{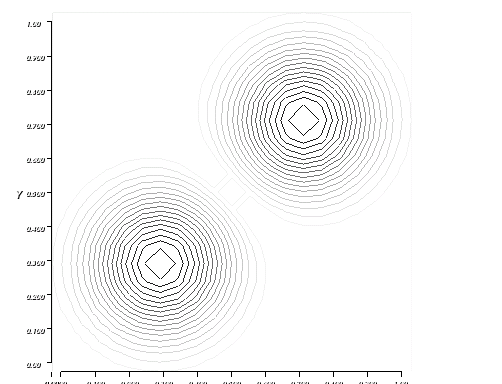}}\qquad
\subfigure[$k=0.06$]%
{\includegraphics[scale=0.33]{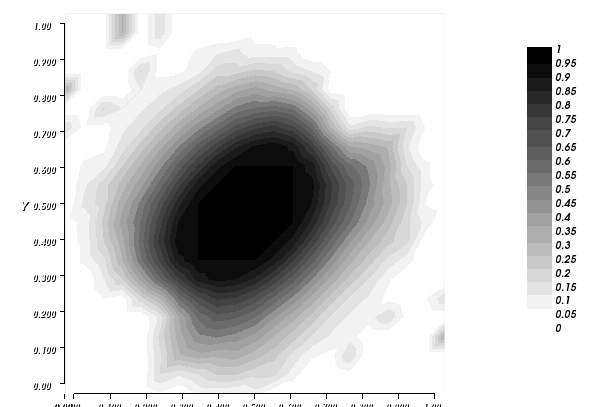}}\qquad
\subfigure[$k=0.4$]%
{\includegraphics[scale=0.33]{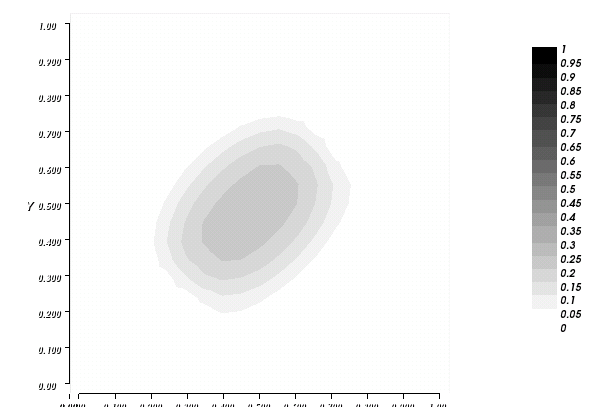}}
\caption{ }\label{fig1}
\end{figure}

\begin{figure}[h!]
\centering%
\subfigure[$\lambda=0.001$]%
{\includegraphics[scale=0.33]{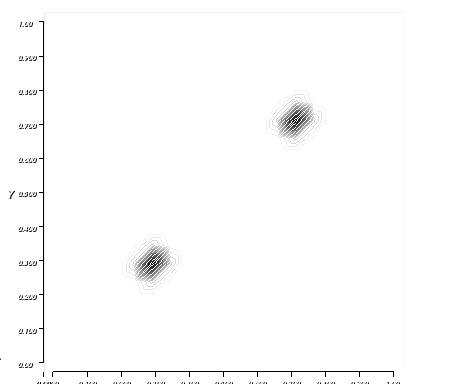}}\qquad
\subfigure[$k=0.01$]%
{\includegraphics[scale=0.33]{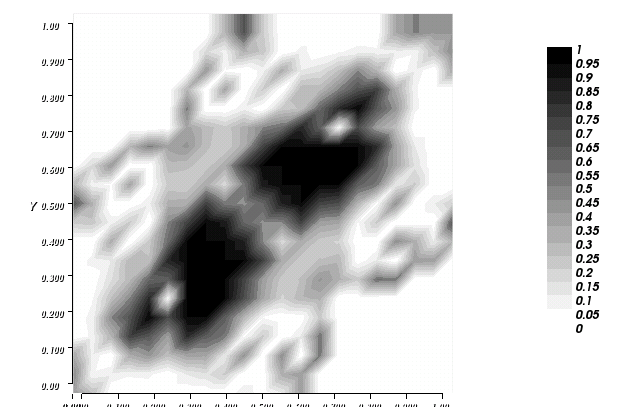}}\qquad
\subfigure[$k=0.05$]%
{\includegraphics[scale=0.33]{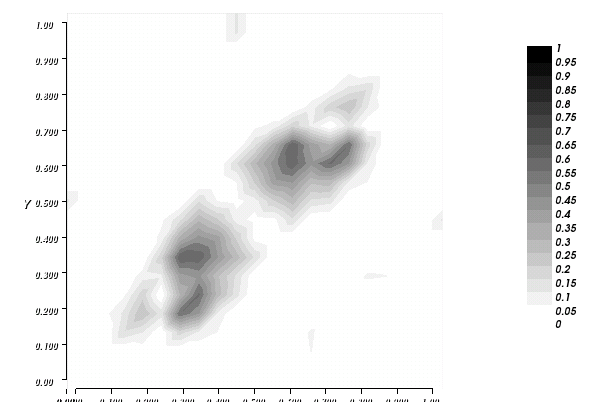}}
\caption{ }\label{fig2}
\end{figure}

\end{exam}

\appendix
\section{A geometric inequality}\label{Aa}

In this appendix we prove the following result.

\begin{prop}\label{appen}
For every set $E\subset\R^d$ and for every $r>0$, setting $E_r=\big\{x\in\R^d\ :\ \dist(x,E)<r\big\}$, we have
\be\label{ineqapp}
\per(E_r)\le\frac dr|E_r|.
\ee
\end{prop}

\begin{proof}
The inequality above can be deduced from the results in the appendix of \cite{busast}; the present proof was obtained during a discussion with Giovanni Alberti, that we thank for his help.

Since the set $E_r$ only depends on the closure of $E$, we may assume that $E$ is closed; moreover, approximating $E$ by smooth sets (for instance by the sets $E_s$ with $s\to0$), we may also assume that $E$ is smooth.

Consider now the function
$$f(r)=d|E_r|-r\per(E_r);$$
proving \eqref{ineqapp} amounts to show that $f(r)\ge0$ for every $r>0$. Since $E$ is assumed smooth, we have
$$\lim_{r\to0}|E_r|=|E|,\qquad\lim_{r\to0}\per(E_r)=\per(E),$$
so that
$$\lim_{r\to0}f(r)=d|E|\ge0.$$
By the coarea formula we have for all $r<s$
$$|E_s|-|E_r|=\int_{E_s\setminus E_r}|\nabla\dist(x,E)|\,dx=\int_r^s\per(E_t)\,dt$$
so that
$$\frac{d}{dr}|E_r|=\per(E_r).$$
Denoting by $h(x)$ the mean curvature of $\partial E_r$ at $x$, and taking into account the definition of $E_r$, we have $h(x)\le(d-1)/r$, so that
$$\frac{d}{dr}\per(E_r)=\int_{\partial E_r}h(x)\,d\HH^{d-1}\le\frac{d-1}{r}\per(E_r).$$
Therefore,
$$f'(r)=d\frac{d}{dr}|E_r|-\per(E_r)-r\frac{d}{dr}\per(E_r)\ge0,$$
which implies that $f(r)\ge0$ for every $r>0$.
\end{proof}

\ack The work of the first and third authors is part of the project 2010A2TFX2 {\it``Calcolo delle Variazioni''} funded by the Italian Ministry of Research and University. The second author gratefully acknowledges the support of INRIA and the ANR through the Projects ISOTACE (ANR-12-MONU-0013) and OPTIFORM (ANR-12-BS01-0007).


\end{document}